\documentclass{amsart}

\newtheorem{theorem}{Theorem}[section]

\theoremstyle{definition}
\newtheorem{definition}[theorem]{Definition}
\newtheorem{example}[theorem]{Example}

\theoremstyle{remark}
\newtheorem{remark}[theorem]{Remark}

\numberwithin{equation}{section}
\def\R{\mathbb{R}}

\begin{document}

\title{Intrinsic representation curves}

\author{Hector Efren Guerrero Mora}
\address{Department of Mathematics, Cauca University, Cauca, Colombia}
\email{heguerrero@unicauca.edu.co}
\thanks{The first author was supported in part by Cauca university project ID 4558.}


\date{December 5, 2017.}
\dedicatory{This paper is dedicated to my brother.}


\begin{abstract}
The purpose of this article is to find a family of curves parametrized by arc length and that depend on an angular function and an intrinsic fraction function, which is defined as the quotient between torsion and curvature.
We find for this family of curves explicit formulas of curvature, torsion and geodetic curvature, in terms of the angular function and the intrinsic fraction function.
Applications are found for the case of the general helices and slant helices.
\\
\end{abstract}

\maketitle
\section*{Introduction}
\begin{definition}
 Let $\alpha:I\rightarrow \R^3$ be a regular curve in $\R^3$, which is parametrized by arc length $s$ and $\kappa=\kappa(s)$, $\tau=\tau(s)$ the curvature and torsion respectively.\\ The funtion
 \begin{equation*}
 \frac{\tau}{\kappa}(s)=\frac{\tau(s)}{\kappa(s)},
 \end{equation*}is called intrinsic fraction function of the curve $\alpha$
 \end{definition}
 \begin{remark}
  Let $\alpha\in \mathcal{C}$ and suppose that its intrinsic fraction function $\varphi$ is not constant. Consider its tangent vector in spherical coordinates, this is
\begin{equation*}
\textbf{t}(s) =(\sin \phi\cos \theta,\sin \phi\sin \theta,\cos \phi),
\end{equation*}  where $\theta=\theta(s)$  is a differential function of angle in the $xy-$ plane from the positive $x-$axis and anticlockwise and $\phi=\phi(s)$ is a differential function of angle from the positive $z-$ axis. we need to find explicit expressions of curves that have an intrinsic fraction function $\varphi$
\\
Since the general Serret-Frenet equation for a curve in space
\begin{eqnarray*}
\frac{d}{ds}\textbf{t}&=&\kappa\textbf{n}\\
\frac{d}{ds}\textbf{n}&=&-\kappa\textbf{t}+\tau\textbf{b}\\
\frac{d}{ds}\textbf{b}&=&-\tau\textbf{n},
\end{eqnarray*}where $\kappa$ and $\tau$ are the curvature and torsion and $s$ the arc length at the point with tangent, normal and binormal $\textbf{t},\textbf{n}$ and $\textbf{b}$.
\\
Differentiating with respect to $s$ the tangent vector $\textbf{t}(s)$
\begin{equation*}\frac{d\textbf{t}}{ds}=
(\phi'\cos \phi\cos \theta-\theta'\sin \phi\sin \theta, \phi'\cos \phi\sin \theta +\theta'\sin\phi\cos \theta,-\phi'\sin \phi)\end{equation*}
\\ and calculating its norm we obtain the formula for the curvature $\kappa$ in terms of $\phi$ and $\theta$
 \begin{equation*}
 \mid\mid\frac{d\textbf{t}}{ds}\mid \mid =\kappa=\sqrt{(\phi')^2+(\theta')^2\sin^2\phi}.
 \end{equation*}
 Differentiating with respect to $s$ the binormal vector
\begin{equation*}
\textbf{b}=(\frac{-\phi'\sin \theta}{\kappa}-\frac{\theta'\sin 2\phi\cos\theta}{2\kappa},\frac{-\phi'\cos \theta}{\kappa}-\frac{\theta'\sin 2\phi\sin \theta}{2\kappa},\frac{\theta'\sin^2 \phi}{\kappa}),
\end{equation*}
  and from Serret-Frenet equation one can derive the formulas: \begin{eqnarray*}
  &&\frac{d}{ds}<\textbf{b},\textbf{e}_1>=\frac{d}{ds}(\frac{-\phi'\sin \theta}{\kappa}-\frac{\theta'\sin 2\phi\cos\theta}{2\kappa})=\\
  &&-(\frac{\phi'\theta''\sin \phi+2(\phi')^2\theta'\cos \phi+(\theta')^3\sin^2 \phi\cos \phi-\phi''\theta'\sin \phi}{\kappa^2})\frac{(\phi'\cos \phi\cos \theta-\theta'\sin \phi\sin \theta)}{\kappa}\\
  &&\frac{d}{ds}<\textbf{b},\textbf{e}_2>=\frac{d}{ds}(\frac{-\phi'\cos \theta}{\kappa}-\frac{\theta'\sin 2\phi\sin \theta}{2\kappa})=\\
  &&-(\frac{\phi'\theta''\sin\phi+2(\phi')^2\theta'\cos\phi+(\theta')^3\sin^2 \phi\cos \phi-\phi''\theta'\sin \phi}{\kappa^2})\frac{(\phi'\cos\phi\sin \theta+\theta'\sin \phi\cos \theta)}{\kappa}\\
   &&\frac{d}{ds}<\textbf{b},\textbf{e}_3>=\frac{d}{ds}(\frac{\theta'\sin^2 \phi}{\kappa})=\\
  &&(\frac{\phi'\theta''\sin \phi+2(\phi')^2\theta'\cos \phi+(\theta')^3\sin^2 \phi\cos\phi-\phi''\theta'\sin \phi}{\kappa^2})\frac{(\phi'\sin \phi)}{\kappa}\\
  \end{eqnarray*}
  This is
  \begin{equation*}
  \frac{d}{ds}<\textbf{b},\textbf{e}_1>=-\tau\frac{(\phi'\cos \phi\cos \theta-\theta'\sin \phi\sin \theta)}{\kappa}
  \end{equation*}
  \begin{equation*}\frac{d}{ds}<\textbf{b},\textbf{e}_2>=-\tau\frac{(\phi'\cos\phi\sin \theta+\theta'\sin \phi\cos \theta)}{\kappa}
  \end{equation*}
  \begin{equation}\label{Formula 1}\frac{d}{ds}<\textbf{b},\textbf{e}_3>=\tau\frac{(\phi'\sin \phi)}{\kappa}
  \end{equation}
 One obtains the fundamental equations:
\begin{eqnarray*}
\frac{d}{ds}(\frac{-\phi'\sin \theta}{\kappa}-\frac{\theta'\sin 2\phi\cos\theta}{2\kappa})&=&-\frac{\tau}{\kappa}(\phi'\cos \phi\cos \theta-\theta'\sin \phi\sin \theta)\\
\frac{d}{ds}(\frac{-\phi'\cos \theta}{\kappa}-\frac{\theta'\sin 2\phi\sin \theta}{2\kappa})&=&-\frac{\tau}{\kappa}(\phi'\cos\phi\sin \theta+\theta'\sin \phi\cos \theta)\\
\frac{d}{ds}(\frac{\theta'\sin^2 \phi}{\kappa})&=&
  \frac{\tau}{\kappa}(\phi'\sin \phi),
\end{eqnarray*}
Putting $\frac{\tau}{\kappa}(s)=\varphi(s)=\varphi$. Then by immediate integration takes the form
\begin{equation}\label{formula fundamental}
\frac{\theta'\sin^2 \phi}{\kappa}=\frac{\theta'\sin^2 \phi}{\sqrt{(\phi')^2+(\theta')^2\sin^2\phi}}=\int\varphi\phi'\sin \phi ds.
\end{equation}
It follows that
\begin{equation*}
\frac{(\theta')^2\sin ^4 \phi}{(\phi')^2+(\theta')^2\sin^2\phi}=(\int\varphi\phi'\sin \phi ds)^2.
\end{equation*}
Consider the case where $\cos \phi$ is different from a constant, that is, $\frac{d}{ds}\cos \phi\neq 0$, thus $\phi'\neq 0$ and
\begin{equation*}
(\theta')^2=\frac{(\phi')^2(\int \varphi\phi'\sin \phi ds)^2}{\sin^2 \phi(\sin^2\phi-(\int \varphi\phi'\sin \phi ds)^2)}.
\end{equation*}
the above motivates us to give the following definition
 \end{remark}
\begin{definition}
Let $\varphi$ be a defined differentiable function of an open interval $I$ with real values,\begin{equation*} \mathcal{D}_{\varphi}=\{ \phi \mid \phi :I_{\mbox{\tiny$\phi$}}\subset I\rightarrow \R \  \text{is differentiable,}\  \ \ 1>\cos^2\phi+(\int\varphi(\cos\phi)' ds)^2\ \text{and}\   (\cos\phi)' <\ 0 \}.\end{equation*}Then the curves defined as
\begin{multline*}
   \rho_{\varphi}(\phi)(s)=(\int{\sin \phi \cos (\int{\frac{\phi'\csc\phi(\int{\varphi\phi'\sin\phi ds)}}{\sqrt{\sin^2\phi-(\int\varphi\phi'\sin\phi ds)^2}}}}ds)ds ,\\ \int{\sin \phi \sin (\int{\frac{\phi'\csc\phi(\int{\varphi\phi'\sin\phi ds)}}{\sqrt{\sin^2\phi-(\int\varphi\phi'\sin\phi ds)^2}}}}ds)ds ,\int{\cos \phi ds}),\ \text{for some} \ \phi\in \mathcal{D}_{\varphi}
 \end{multline*}
 are called intrinsic representation curves.
\end{definition}
\begin{example}
 The set  $\mathcal{D}_{\varphi}$ is different from empty, for each differentiable function $\varphi:I=(a,b)\rightarrow R$.
In effect, if function $\varphi(s)=\varphi_0$ is a constant, then define $\phi$ on $I=(a,b)$ as
\begin{equation*}
\phi(s)=\arccos{(\frac{\cos{(\frac{s-a}{b-a})\pi}}{1+\varphi^2_0})},
\end{equation*}It is clear that
\begin{equation*}
\frac{d}{ds}\cos\phi(s)=\frac{-\pi\sin{(\frac{s-a}{b-a})\pi}}{(b-a)(1+\varphi^2_0)}<0,
\end{equation*} and
\begin{eqnarray*}
1-\cos^2\phi-(\int\varphi_0(\cos\phi)' ds)^2&=&1-(1+\varphi^2_0)\cos^2\phi\\
&=&1-\frac{1}{1+\varphi^2_0}\cos^2{(\frac{s-a}{b-a})\pi}>0,
\end{eqnarray*}
thus $\mathcal{D}_{\varphi}\neq \emptyset$ and since
\begin{eqnarray*}
&&\int{\frac{\phi'\csc\phi(\int{\varphi\phi'\sin\phi ds)}}{\sqrt{\sin^2\phi-(\int\varphi\phi'\sin\phi ds)^2}}}ds=\\
&&\frac{-\varphi_0\pi\sqrt{1+\varphi^2_0}}{(b-a)\sqrt{2}}
\int{\frac{\sin{\pi(\frac{s-a}{b-a})}\cos{\pi(\frac{s-a}{b-a})}}{((1+\varphi^2_0)^2-\cos^2{\pi(\frac{s-a}{b-a})})\sqrt{1+2\varphi^2_0-\cos{2\pi(\frac{s-a}{b-a})}}}}=\\
&&-\arctan{\frac{\sqrt{1+2\varphi^2_0-\cos{2\pi(\frac{s-a}{b-a})}}}{\varphi_0\sqrt{2}\sqrt{1+\varphi^2_0}}}.
\end{eqnarray*}
And a direct calculation shows that the intrinsic representation curves are
\begin{multline*}
   \rho_{\varphi}(\phi)(s)=(\int{\sin \phi \cos (\int{\frac{\phi'\csc\phi(\int{\varphi\phi'\sin\phi ds)}}{\sqrt{\sin^2\phi-(\int\varphi\phi'\sin\phi ds)^2}}}}ds)ds ,\\ \int{\sin \phi \sin (\int{\frac{\phi'\csc\phi(\int{\varphi\phi'\sin\phi ds)}}{\sqrt{\sin^2\phi-(\int\varphi\phi'\sin\phi ds)^2}}}}ds)ds ,\int{\cos \phi ds})\\=(\frac{\varphi_0}{\sqrt{1+\varphi^2_0}}s,-\frac{1}{\sqrt{2}(1+\varphi^2_0)}\int{\sqrt{1+2\varphi^2_0-\cos{2\pi(\frac{s-a}{b-a})}}}ds,
   \frac{(b-a)\sin{\pi(\frac{s-a}{b-a})}}{(1+\varphi^2_0)\pi}).
 \end{multline*}
 Note that these curves are general helices,
using the formulas of curvature \begin{equation*}\kappa=\frac{\mid\mid \beta'\wedge \beta'' \mid\mid}{\mid \mid \beta' \mid \mid^3}\end{equation*} and torsion \begin{equation*}\tau=\frac{\beta'\wedge \beta''\cdot \beta'''}{\mid\mid\beta'\wedge\beta''\mid \mid^2},\end{equation*}
we obtain that
\begin{eqnarray*}
\kappa&=&\frac{\pi\sin{\frac{\pi(s-a)}{b-a}}}{(b-a)\sqrt{1+\varphi_{0}^2}\sqrt{(\varphi_{0}^2+\sin^2{\frac{\pi(s-a)}{b-a}})}}\\
\tau&=&\frac{\sqrt{2}\varphi_0\pi\sin{\frac{\pi(s-a)}{b-a}}}{(b-a)\sqrt{1+\varphi^2_0}\sqrt{1+2\varphi^2_0-\cos{\frac{2\pi(s-a)}{b-a}}}}\\
&=&\frac{\varphi_0\pi\sin{\frac{\pi(s-a)}{b-a}}}{(b-a)\sqrt{1+\varphi_{0}^2}\sqrt{(\varphi_{0}^2+\sin^2{\frac{\pi(s-a)}{b-a}})}},
\end{eqnarray*}thus $\frac{\tau}{\kappa}=\varphi_0.$
\end{example}
Now, the theorem of Hector's intrinsic representation curves will be demonstrated.
\begin{theorem}\label{representacion}
Let $\rho_{\varphi}(\phi)$ be a intrinsic representation curve, for some $\phi\in \mathcal{D}_{\varphi}$. Then
\begin{enumerate}
\item The curvature of  $\rho_{\varphi}(\phi)$ is
\begin{equation*}
 \kappa_{\rho}=\frac{ \phi' \sin  \phi}{\sqrt{\sin^2\phi-(\int\varphi\phi'\sin\phi ds)^2}}
\end{equation*}
\item The torsion is
\begin{equation*}
\tau_{\rho}=\frac{\varphi \phi' \sin  \phi}{\sqrt{\sin^2\phi-(\int\varphi\phi'\sin\phi ds)^2}}.
\end{equation*}
\item The intrinsic fraction function is
\begin{equation*}
\frac{\tau_{\rho}}{\kappa_{\rho}}(s)=\varphi(s),
\end{equation*}for all $s\in I_{\mbox{\tiny$\phi$}}$.
\item The geodesic curvature of the normal vector $n$ of the curve $\rho_{\varphi}(\phi_i)$ in $S^2$ is given by
\begin{equation*}
\sigma_{\rho}=\frac{\varphi'\sqrt{\sin^2\phi-(\int\varphi\phi'\sin\phi ds)^2}}{(1+\varphi^2)^{3/2}\phi' \sin  \phi}
\end{equation*}
\end{enumerate}
\end{theorem}
\begin{proof}
It is clear that $\rho_{\varphi}(\phi)\in \mathcal{C}= \{\alpha \mid \alpha:I\rightarrow \R^3, \text{is a curve parametrized by arc length $s$}\}$ and by a direct calculation, using the formulas of curvature \begin{equation*}\kappa=\frac{\mid\mid \beta'\wedge \beta'' \mid\mid}{\mid \mid \beta' \mid \mid^3}\end{equation*} and torsion \begin{equation*}\tau=\frac{\beta'\wedge \beta''\cdot \beta'''}{\mid\mid\beta'\wedge\beta''\mid \mid^2},\end{equation*} we obtain that
 \begin{eqnarray*}
 \kappa&=&\frac{\mid \phi' \sin  \phi\mid}{\sqrt{\sin^2\phi-(\int\varphi\phi'\sin\phi ds)^2}}=\frac{ \phi' \sin  \phi}{\sqrt{\sin^2\phi-(\int\varphi\phi'\sin\phi ds)^2}}\\
 \tau&=&\frac{\varphi \phi' \sin  \phi}{\sqrt{\sin^2\phi-(\int\varphi\phi'\sin\phi ds)^2}}.
 \end{eqnarray*} Therefore the intrinsic fraction function of curve  $\rho_{\varphi}(\phi)$ is given by
 \begin{equation*}
 \frac{\tau_{\rho}}{\kappa_{\rho}}=\varphi.
 \end{equation*}
 By a straightforward calculation, the geodesic curvature of the normal vector $n$ of the curve $\rho_{\varphi}(\phi_i)$ in $S^2$ is given by
 \begin{eqnarray*}
 \sigma_{\rho}&=&\frac{(\frac{\tau_{\rho}}{\kappa_{\rho}})'}{(1+(\frac{\tau_{\rho}}{\kappa_{\rho}})^2)^{3/2}\kappa_{\rho}}\\
 &=&\frac{\varphi'\sqrt{\sin^2\phi-(\int\varphi\phi'\sin\phi ds)^2}}{(1+\varphi^2)^{3/2}\phi' \sin  \phi}.
 \end{eqnarray*}
\end{proof}
We now show some applications of the theorem of Hector's intrinsic representation curves
\begin{theorem}
If the curve $\alpha:I=(a,b)\rightarrow \R^3$, parameterized by length of arc $s$, is a general helix, then there exists an interval $J\subset I$ such that the restriction from $\alpha$ to $J$ coincides with curve
\begin{equation*}
\beta(s)=(\frac{1}{\sqrt{(1+\varphi^2_0}}\int{\sqrt{1-(1+\varphi_0^2)\cos^2\xi}}ds,\frac{\mid\varphi_0 \mid s}{\sqrt{1+\varphi_0^2}},\int{\cos \xi}ds),
\end{equation*}for all $s\in J$, where $\varphi_0\neq 0$ is a constant and\\ \begin{equation*}\xi\in \mathcal{D}_{\varphi_0}=\{ \phi \mid \phi:I_{\mbox{\tiny$\phi$}}\subset I\rightarrow \R\  \text{is differentiable,}\ \ 1>\cos^2\phi+(\int\varphi_{0}(\cos\phi)' ds)^2\ \text{and}\   (\cos\phi)' <\ 0 \},\end{equation*} or differ from a rigid movement.
\\if conversely curve $\beta$ is defined as above, then $\beta$ is a general helix.
\end{theorem}

\begin{proof}
Suppose that the curve $\alpha:I\rightarrow \R^3$ parameterized by arc length $s$ is a general helix, and $\kappa_{\alpha}=\kappa_{\alpha}(s)$ is its curvature function and $\tau_{\alpha}=\tau_{\alpha}(s)$ is its torsion function. Then by Lancret's theorem we know that $\frac{\tau_{\alpha}}{\kappa_{\alpha}}(s)=\varphi_0$ is a constant.\\
By Theorem \ref{representacion}  we can find $\phi$ such that the curvature of $\rho_{\varphi_0}(\phi)$ coincides with the curvature of $\alpha$, in effect, we can write
\begin{equation*}
\kappa_{\alpha}=\kappa_{\rho}=\frac{ \phi' \sin  \phi}{\sqrt{\sin^2\phi-(\int\varphi_0\phi'\sin\phi ds)^2}}.
\end{equation*}This implies
\begin{equation*}
\int{\kappa_{\alpha}}ds=-\frac{\arcsin{(\sqrt{1+\varphi_0^2}\cos \phi)}}{\sqrt{1+\varphi^2_0}},
\end{equation*}
Therefore, we can define
\begin{equation*}
\xi(s)=\arccos{(\frac{\sin{(\sqrt{1+\varphi^2_0}\int^s_a{\kappa_{\alpha}}ds})}{\sqrt{1+\varphi^2_0}})},
\end{equation*}in the interval $J=I_{\mbox{\tiny$\xi$}}=(a,i)\subset I$, where $\int^s_a{\kappa_{\alpha}}ds<\frac{\pi}{2\sqrt{1+\varphi^2_0}}$.
\\Then we have $\xi\in\mathcal{D}_{\varphi_0}$, the curvature of $\rho_{\varphi_0}(\xi)$ is $\kappa_{\alpha}=\kappa_{\alpha}(s)$ and since the intrinsic fraction function of $\rho_{\varphi_0}(\xi)$ is $\frac{\tau_{\rho}}{\kappa_{\rho}}=\varphi_0=\frac{\tau_{\alpha}}{\kappa_{\alpha}}$, we conclude that the torsion of  $\rho_{\varphi_0}(\xi)$ is equal to that of the curve $\alpha$.\\
Now, note that the intrinsic representation curve $\rho_{\varphi_0}(\xi)$ is given by
\begin{multline*}
   \rho_{\varphi_0}(\xi)(s)=(\int{\sin \xi \cos (\int{\frac{\xi'\csc\xi(\int{\varphi_0\xi'\sin\xi ds)}}{\sqrt{\sin^2\xi-(\int\varphi_0\xi'\sin\xi ds)^2}}}}ds)ds ,\\ \int{\sin \xi \sin (\int{\frac{\xi'\csc\xi(\int{\varphi_0\xi'\sin\xi ds)}}{\sqrt{\sin^2\xi-(\int\varphi_0\xi'\sin\xi ds)^2}}}}ds)ds ,\int{\cos \xi ds}),\ \text{where} \ \xi\in \mathcal{D}_{\varphi_0}.
 \end{multline*}
Note that
\begin{eqnarray*}
\int{\frac{\xi'\csc\xi(\int{\varphi_0\xi'\sin\xi ds)}}{\sqrt{\sin^2\xi-(\int\varphi_0\xi'\sin\xi ds)^2}}}ds&=&-\int{\frac{\xi'\cos\xi}{\sin\xi\sqrt{-1+\frac{(1+\varphi_0^2)}{\varphi_0^2}\sin^2\xi
}}}ds\\&=&\arctan(\frac{1}{\sqrt{-1+\frac{(1+\varphi_0^2)}{\varphi_0^2}\sin^2\xi }}),
\end{eqnarray*}
Therefore, \begin{eqnarray*}
\cos (\int{\frac{\xi'\csc\xi(\int{\varphi_0\xi'\sin\xi ds)}}{\sqrt{\sin^2\xi-(\int\varphi_0\xi'\sin\xi ds)^2}}}ds)&=&\cos(\arctan(\frac{1}{\sqrt{-1+\frac{(1+\varphi_0^2)}{\varphi_0^2}\sin^2\xi }}))\\
&=&\frac{\sqrt{-(\frac{\varphi_0^2}{1+\varphi_0^2})+\sin^2\xi }}{\sin\xi}\\
&=&\frac{\sqrt{1-(1+\varphi_0^2)\cos^2\xi}}{\sqrt{(1+\varphi^2_0}\sin\xi}\\
\sin (\int{\frac{\xi'\csc\xi(\int{\varphi_0\xi'\sin\xi ds)}}{\sqrt{\sin^2\xi-(\int\varphi_0\xi'\sin\xi ds)^2}}}ds)&=&\sin(\arctan(\frac{1}{\sqrt{-1+\frac{(1+\varphi_0^2)}{\varphi_0^2}\sin^2\xi}}))\\
&=&\frac{\frac{\mid\varphi_0\mid}{\sqrt{1+\varphi_0^2}}}{\sin\xi}.
\end{eqnarray*}
\\
Consequently the intrinsic representation curve is given by
\begin{eqnarray*}
\rho_{\varphi_0}(\xi)(s)&=&(\int{\frac{\sqrt{1-(1+\varphi_0^2)\cos^2\xi}}{\sqrt{(1+\varphi^2_0}}}ds,\int{\frac{\mid\varphi_0 \mid}{\sqrt{1+\varphi_0^2}}}ds,\int{\cos \xi}ds)\\
&=&(\frac{1}{\sqrt{(1+\varphi^2_0}}\int{\sqrt{1-(1+\varphi_0^2)\cos^2\xi}}ds,\frac{\mid\varphi_0 \mid s}{\sqrt{1+\varphi_0^2}},\int{\cos \xi}ds).
\end{eqnarray*}
 Therefore  $\rho_{\varphi_0}(\xi)=\alpha $, in the interval $J$, except for a rigid movement.
 \\Reciprocally, suppose that
  \begin{equation*}
\alpha(s)=(\frac{1}{\sqrt{(1+\varphi^2_0}}\int{\sqrt{1-(1+\varphi_0^2)\cos^2\xi}}ds,\frac{\mid\varphi_0 \mid s}{\sqrt{1+\varphi_0^2}},\int{\cos \xi}ds),
\end{equation*}for all $s\in I_{\mbox{\tiny$\xi$}}$, where $\varphi_0\neq 0$ is a constant and\\ \begin{equation*}\xi\in \mathcal{D}_{\varphi_0}=\{ \phi \mid \phi:I_{\mbox{\tiny$\phi$}}\subset I\rightarrow \R\  \text{is differentiable,}\ \ 1>\cos^2\phi+(\int\varphi_0(\cos\phi)' ds)^2\ \text{and}\   (\cos\phi)' <\ 0 \},\end{equation*}
Note that: This curve is parameterized by arc length $s$ and the curvature function $\kappa=\kappa(s)$ is given by
\begin{equation}\label{curvatura de la helice}
\kappa=\frac{\mid\mid \alpha'\wedge \alpha'' \mid\mid}{\mid \mid \alpha' \mid \mid^3}=\frac{ \xi' \sin  \xi}{\sqrt{1-(1+\varphi_0^2)\cos^2\xi}}
 \end{equation}
 and the torsion function $\tau=\tau(s)$ is
 \begin{equation*}
 \tau=\frac{\alpha'\wedge \alpha''\cdot \alpha'''}{\mid\mid\alpha'\wedge\alpha''\mid \mid^2}=\frac{\mid\varphi_0 \mid \xi' \sin  \xi}{\sqrt{1-(1+\varphi_0^2)\cos^2\xi}}.
 \end{equation*} Therefore the function intrinsic fraction is $\frac{\tau}{\kappa}(s)=\mid\varphi_0 \mid$, that is, the curve $\alpha$ is a general helix.
\end{proof}
\begin{remark}
The expression
\begin{equation*}
\alpha(s)=(\frac{1}{\sqrt{(1+\varphi^2_0}}\int{\sqrt{1-(1+\varphi_0^2)\cos^2\xi}}ds,\frac{\mid\varphi_0 \mid s}{\sqrt{1+\varphi_0^2}},\int{\cos \xi}ds),
\end{equation*}where \begin{equation*} \xi\in \mathcal{D}_{\varphi}=\{ \phi \mid \phi :I_{\mbox{\tiny$\phi$}}\subset I\rightarrow \R \  \text{is differentiable,}\  \ \ 1>\cos^2\phi+(\int\varphi(\cos\phi)' ds)^2\ \text{and}\   (\cos\phi)' <\ 0 \}.\end{equation*}
 coincides with the classic expression known from the general helix.
\\In effect, from the curvature \ref{curvatura de la helice} of $\alpha$  it can be deduced that
\begin{equation*}
\int{\kappa_{\alpha}ds}=-\frac{\arcsin{(\sqrt{1+\varphi^2_0}\cos\xi)}}{\sqrt{1+\varphi^2_0}},
\end{equation*}so we have, $-\frac{\pi}{2}\leq \sqrt{1+\varphi^2_0}\int{\kappa_{\alpha}ds}\leq \frac{\pi}{2}$ and \begin{eqnarray*}\sin(\sqrt{1+\varphi^2_0}\int{\kappa_{\alpha}ds})&=&-\sqrt{1+\varphi^2_0}\cos\xi\\
\cos(\sqrt{1+\varphi^2_0}\int{\kappa_{\alpha}ds})&=&\sqrt{1-(1+\varphi^2_0)\cos^2\xi},\end{eqnarray*}
then we have the classic expression of the general helix, given by
\begin{equation*}
\alpha(s)=(\frac{1}{\sqrt{(1+\varphi^2_0}}\int{\cos(\sqrt{1+\varphi^2_0}\int{\kappa_{\alpha}ds})}ds,\frac{\mid\varphi_0 \mid s}{\sqrt{1+\varphi_0^2}},\frac{-1}{\sqrt{(1+\varphi^2_0}}\int{\sin(\sqrt{1+\varphi^2_0}\int{\kappa_{\alpha}ds})}ds)
\end{equation*}
\end{remark}
Here is another example, that related to the curves denominated slant helix
\begin{theorem}
The curve $\alpha$ is a slant helixe if and only if
\begin{multline*}
\alpha(s)=(\int{\sqrt{\frac{1+m^2+m^2\varphi^2}{(1+m^2)(1+\varphi^2)}}}\cos{[\frac{\sqrt{1+m^2}\arctan{(\varphi})}{ m}}-\arctan{(\frac{ m\varphi}{\sqrt{1+m^2}})]}ds,\\\pm\int{\sqrt{\frac{1+m^2+m^2\varphi^2}{(1+m^2)(1+\varphi^2)}}}\sin{[\frac{\sqrt{1+m^2}\arctan{(\varphi)}}{ m}}-\arctan{(\frac{ m\varphi}{\sqrt{1+m^2}})]}ds,\\ \pm\int{\frac{\varphi}{\sqrt{(1+m^2)(1+\varphi^2)}}}ds),
\end{multline*}where $m$ is a constant, $m> 0$ and $\varphi=\varphi(s)$ is a differentiable function such that\\ $\varphi'(s)>0$.(respectively $m<0$ and $\varphi'(s)<0$). Any other slant helixe differs from $\alpha$ by a rigid motion.
\end{theorem}
\begin{proof}
Suppose that the curve $\alpha:I\rightarrow \R^3$ parameterized by arc length $s$ is a general slant helix, and $\kappa_{\alpha}=\kappa_{\alpha}(s)$ is its curvature function and $\tau_{\alpha}=\tau_{\alpha}(s)$ is its torsion function. Then by Izumiya and Takeuchi´ theorem [4] we know that the geodesic curvature of the principal normal of the curve $\alpha$ is a constant function, this is
\begin{equation*}
\frac{\kappa^2_{\alpha}}{(\kappa^2_{\alpha}+\tau^2_{\alpha})^{3/2}}(\frac{\tau_{\alpha}}{\kappa_{\alpha}})'(s)=m
\end{equation*}
 Defining  $\varphi(s)=\frac{\tau_{\alpha}(s)}{\kappa_{\alpha}(s)},$ if $m>0$, then we have $\varphi'(s)>0$. Now, by the theorem \ref{representacion}  we can find $\xi$ such that the geodesic curvature of the normal vector $n$ of the curve $\rho_{\varphi}(\xi)$ in $S^2$  coincides with the geodesic curvature of the normal vector $n$ of the curve $\alpha$ in $S^2$. In fact, consider the geodesic curvature of the normal vector of curve $\rho_{\varphi}(\phi)$, then
\begin{equation*}
\sigma=\frac{\varphi'\sqrt{\sin^2\phi-(\int\varphi\phi'\sin\phi ds)^2}}{(1+\varphi^2)^{3/2}\phi' \sin  \phi}=m,
\end{equation*}
it implies
\begin{equation*}
\frac{\varphi'}{(1+\varphi^2)^{3/2}}=m\frac{\phi' \sin  \phi}{\sqrt{\sin^2\phi-(\int\varphi\phi'\sin\phi ds)^2}}.
\end{equation*}
The function $\xi$ defined by $\cos\xi=\frac{-\varphi}{\sqrt{(1+m^2)(1+\varphi^2)}}$ satisfies the above equation, since
\begin{eqnarray*}
\xi'\sin \xi&=&\frac{\varphi'}{\sqrt{1+m^2}(1+\varphi^2)^{3/2}}\\
\sqrt{\sin^2\xi-(\int\varphi\xi'\sin\xi ds)^2}&=&\frac{m}{\sqrt{1+m^2}},
\end{eqnarray*}and as
\begin{equation*} \xi \in\mathcal{D}_{\varphi}=\{ \phi \mid \phi :I_{\mbox{\tiny$\phi$}}\subset I\rightarrow \R \  \text{is differentiable,}\  \ \ 1>\cos^2\phi+(\int\varphi(\cos\phi)' ds)^2\ \text{and}\   (\cos\phi)' <\ 0 \},\end{equation*}
 then we have
\begin{multline*}
   \rho_{\varphi}(\xi)(s)=(\int{\sin \xi \cos (\int{\frac{\xi'\csc\xi(\int{\varphi\xi'\sin\xi ds)}}{\sqrt{\sin^2\xi-(\int\varphi\xi'\sin\xi ds)^2}}}}ds)ds ,\\ \int{\sin \xi \sin (\int{\frac{\xi'\csc\xi(\int{\varphi\xi'\sin\xi ds)}}{\sqrt{\sin^2\xi-(\int\varphi\xi'\sin\xi ds)^2}}}}ds)ds ,\int{\cos \xi ds})\\=(\int{\sqrt{\frac{1+m^2+m^2\varphi^2}{(1+m^2)(1+\varphi^2)}}}\cos{[\frac{\sqrt{1+m^2}}{m}\int{\frac{\varphi'}{(1+m^2+m^2\varphi^2)(1+\varphi^2)}}ds]}ds,
   \\-\int{\sqrt{\frac{1+m^2+m^2\varphi^2}{(1+m^2)(1+\varphi^2)}}}\sin{[\frac{\sqrt{1+m^2}}{m}\int{\frac{\varphi'}{(1+m^2+m^2\varphi^2)(1+\varphi^2)}}ds]}ds,\\ -\int{\frac{\varphi}{\sqrt{(1+m^2)(1+\varphi^2)}}}ds) \\=(\int{\sqrt{\frac{1+m^2+m^2\varphi^2}{(1+m^2)(1+\varphi^2)}}}\cos{[\frac{\sqrt{1+m^2}\arctan{(\varphi})}{ m}}-\arctan{(\frac{ m\varphi}{\sqrt{1+m^2}})]}ds,\\-\int{\sqrt{\frac{1+m^2+m^2\varphi^2}{(1+m^2)(1+\varphi^2)}}}\sin{[\frac{\sqrt{1+m^2}\arctan{(\varphi)}}{ m}}-\arctan{(\frac{ m\varphi}{\sqrt{1+m^2}})]}ds,\\ -\int{\frac{\varphi}{\sqrt{(1+m^2)(1+\varphi^2)}}}ds),
 \end{multline*}
 note that the curvature of $\rho_{\varphi}(\xi)$ is given by
 \begin{equation*}
 \kappa_{\rho}=\frac{\varphi'}{m(1+\varphi^2)^{3/2}},
 \end{equation*}now as \begin{equation*}
\frac{\kappa^2_{\alpha}}{(\kappa^2_{\alpha}+\tau^2_{\alpha})^{3/2}}(\frac{\tau_{\alpha}}{\kappa_{\alpha}})'(s)=
\frac{1}{\kappa_{\alpha}(1+\varphi^2)^{3/2}}(\varphi)'(s)=m,
\end{equation*} then we have to $\kappa_{\rho}=\kappa_{\alpha}$.
\\
Since the torsion of $\rho_{\varphi}(\xi)$ is
\begin{equation*}
\tau_{\rho}=\frac{\varphi\varphi'}{m(1+\varphi^2)^{3/2}},
\end{equation*}
then
$\frac{\tau_{\rho}}{\kappa_{\rho}}=\varphi=\frac{\tau_{\alpha}}{\kappa_{\alpha}}$, this is $\tau_{\rho}=\tau_{\alpha}$.
Hence $\alpha=\rho_{\varphi}(\xi)$ or $\alpha$ differs from $\rho_{\varphi}(\xi)$ by a rigid motion.
\\Considering the case $m<0$, the function $\xi$ is defined by $\cos\xi=\frac{\varphi}{\sqrt{(1+m^2)(1+\varphi^2)}}$, where $\varphi=\frac{\tau_{\alpha}}{\kappa_{\alpha}}$
and we arrive at the curves of the form
\begin{multline*}
   \rho_{\varphi}(\xi)(s)=(\int{\sqrt{\frac{1+m^2+m^2\varphi^2}{(1+m^2)(1+\varphi^2)}}}\cos{[\frac{\sqrt{1+m^2}\arctan{(\varphi})}{ m}}-\arctan{(\frac{ m\varphi}{\sqrt{1+m^2}})]}ds,\\\int{\sqrt{\frac{1+m^2+m^2\varphi^2}{(1+m^2)(1+\varphi^2)}}}\sin{[\frac{\sqrt{1+m^2}\arctan{(\varphi)}}{ m}}-\arctan{(\frac{ m\varphi}{\sqrt{1+m^2}})]}ds,\\ \int{\frac{\varphi}{\sqrt{(1+m^2)(1+\varphi^2)}}}ds),
 \end{multline*}
 whose curvature and torsion are: $\frac{\varphi'}{m(1+\varphi^2)^{3/2}}$ and $\frac{\varphi\varphi'}{m(1+\varphi^2)^{3/2}}$, respectively.
 And similarly it is concluded that $\alpha=\rho_{\varphi}(\xi)$ or $\alpha$ differs from $\rho_{\varphi}(\xi)$ by a rigid motion.
 Now, assume that
 \begin{multline*}
\alpha(s)=(\int{\sqrt{\frac{1+m^2+m^2\varphi^2}{(1+m^2)(1+\varphi^2)}}}\cos{[\frac{\sqrt{1+m^2}\arctan{(\varphi})}{ m}}-\arctan{(\frac{ m\varphi}{\sqrt{1+m^2}})]}ds,\\\pm\int{\sqrt{\frac{1+m^2+m^2\varphi^2}{(1+m^2)(1+\varphi^2)}}}\sin{[\frac{\sqrt{1+m^2}\arctan{(\varphi)}}{ m}}-\arctan{(\frac{ m\varphi}{\sqrt{1+m^2}})]}ds,\\ \pm\int{\frac{\varphi}{\sqrt{(1+m^2)(1+\varphi^2)}}}ds),
\end{multline*}
where $m$ is a constant, $m> 0$ and $\varphi=\varphi(s)$ is a differentiable function such that\\ $\varphi'(s)>0$.(respectively $m<0$ and $\varphi'(s)<0$).\\By a direct calculation, using the formulas of curvature \begin{equation*}\kappa=\frac{\mid\mid \beta'\wedge \beta'' \mid\mid}{\mid \mid \beta' \mid \mid^3}\end{equation*} and torsion \begin{equation*}\tau=\frac{\beta'\wedge \beta''\cdot \beta'''}{\mid\mid\beta'\wedge\beta''\mid \mid^2},\end{equation*}
we obtain that
\begin{eqnarray*}
 \kappa_{\alpha}&=&\frac{\varphi'}{m(1+\varphi^2)^{3/2}}\\
 \tau_{\alpha}&=&\frac{\varphi\varphi'}{m(1+\varphi^2)^{3/2}}
\end{eqnarray*} And computing the geodesic curvature of the normal vector $n$ of the curve $\alpha$ in $S^2$, we have
\begin{eqnarray*}
\frac{\kappa^2_{\alpha}}{(\kappa^2_{\alpha}+\tau^2_{\alpha})^{3/2}}(\frac{\tau_{\alpha}}{\kappa_{\alpha}})'&=1.&
\end{eqnarray*}
\end{proof}
\bibliographystyle{amsplain}

\end{document}
